\documentclass[11pt]{article}%
\usepackage{algpseudocode}
\usepackage{float} 
\usepackage{fullpage}
\usepackage{amsfonts}
\usepackage{amsmath}
\usepackage{amssymb}
\usepackage{amsbsy}
\usepackage{amsthm}
\usepackage{mathtools}
\usepackage{epsfig}
\usepackage{graphicx}
\usepackage{colordvi}
\usepackage{graphics}
\usepackage{color}
\usepackage{bm}
\usepackage{verbatim} 
\numberwithin{equation}{section}

\newtheorem{theorem}{Theorem}[section]
\newtheorem{thm}{Theorem}[section]

\newtheorem{lemma}{Lemma}[section]

\newtheorem{algorithm}[thm]{Algorithm}

\newtheorem{remark}[thm]{Remark}

\newcommand{\norm}[1]{\left\Vert#1\right\Vert}
\newcommand{\nr}[1]{\ensuremath{\left\|{#1}\right\|}}

\DeclareMathOperator{\diag}{diag}
\newcommand{\eps}{\varepsilon}

\DeclareMathOperator{\arccosh}{arccosh}

\newcommand{\goto}{\rightarrow}

\newcommand{\bigo}{{\mathcal O}}

\newcommand{\f}{\frac}

\DeclareMathOperator{\spa}{span}

\begin{document}

\title{Momentum accelerated power iterations and the restarted Lanczos method}

\author{
Alessandro Barletta
\thanks{Department of Mathematics, University of Florida, Gainesville, FL
  32611-8105 (jbarletto1@ufl.edu)}
\and 
Nicholas Marshall
\thanks{Department of Mathematics, Oregon State University, Corvallis, OR
  32611-8105 (marsnich@oregonstate.edu)}
\and 
Sara Pollock
\thanks{Department of Mathematics, University of Florida, Gainesville, FL
  32611-8105 (s.pollock@ufl.edu)}
}
\date{\today}

\maketitle

\begin{abstract}
In this paper we compare two methods for finding extremal eigenvalues and eigenvectors: the restarted Lanczos method and momentum accelerated power iterations.  The convergence of both methods is based on ratios of Chebyshev polynomials evaluated at subdominant and dominant eigenvalues; however, the convergence is not the same.  Here we compare the theoretical convergence properties of both methods, and determine the relative regimes where each is more efficient.  We further introduce a preconditioning technique for the restarted Lanczos method using momentum accelerated power iterations, and demonstrate its effectiveness.  The theoretical results are backed up by numerical tests on benchmark problems.
\end{abstract}


\section{Introduction}\label{sec:intro}
The Lanczos algorithm and its variants are among the most efficient methods for approximating an extremal eigenvector, or extremal set of eigenvectors of a large sparse symmetric or complex Hermitian matrix \cite{aishima15,saad80,saad11-book,urschel21}. 
Its implementation is generally limited, however, by computational and memory resources, 
and in practical computing it may be run as a restarted method, as in \cite{sorensen1992implicit,WuSi00}.
As shown in \cite{aishima15}, global convergence properties are maintained by restarted variants of the method, so long as the initial vector is not orthogonal to the target eigenvector.  However, restarting limits the overall efficiency.
On the other hand, recently introduced momentum accelerated power iterations, developed in \cite{XHDMR18}, based on \cite{polyak1964}, further investigated in \cite{RJRH22}, and with a dynamic implementation introduced in \cite{APZ24}, feature lower memory and ease of implementation for the same problem class. 
These methods differ in that the method of \cite{XHDMR18} requires knowledge of the second eigenvalue to set the optimal momentum parameter, while the method of \cite{RJRH22} approximates this eigenvalue with a preliminary deflation stage, and the method of \cite{APZ24} uses a posteriori information based on computed residuals and a Rayleigh quotient to approximate the optimal parameter.
These methods are closely related to Chebyshev iteration techniques as in \cite{GoVa61,GuRo02,knyazev87}, and have the advantages of requiring neither orthogonalization (or re-orthogonaliztion) nor the solution of auxiliary eigenproblems. As shown in \cite{GuRo02}, Chebyshev-type methods are also less sensitive to roundoff errors.

Here we consider the question of determining regimes in which either restarted Lanczos or momentum accelerated power iterations are preferable in terms of convergence rate for the same number of matrix-vector multiplications. For symmetric problems with eigenvalues $|\lambda_1| > |\lambda_2|$, and $|\lambda_{k}| \ge 
|\lambda_{k+1}|$, for $k \ge 2$, we find that the momentum accelerated power iterations achieve faster convergence rates for the same number of matrix-vector multiplies as restarted Lanczos implemented with increasingly large Krylov subspaces as the ratio $|\lambda_2/\lambda_1|$ 
 tends to $1$. We also find that combining the two techniques by using the momentum methods as a preconditioner for restarted Lanczos can lead to faster convergence than either method alone in situations where restarted Lanczos with larger Krylov subspaces yields faster convergence than momentum accelerated power iterations. This technique is distinct from the extrapolation method applied to Arnoldi in \cite{PoSc21-arnoldi}, and may be generalizable to nonsymmetric problems via the more generalized momentum-type methods in \cite{CMP25-d,CMP25-r}.
In what follows, we quantify the regimes where each of the methods considered herein yields faster convergence, and discuss the mechanism behind the successful combination of the two techniques as a momentum-preconditioned restarted Lanczos method.

The remainder of the paper is structured as follows. In subsections \ref{subsec:notation}-\ref{subsec:algo} we state the notation and algorithms used throughout the rest of the paper.
In section \ref{sec:bg} we present background theoretical results on the Chebyshev polynomials, the static and dynamic momentum accelerated power iterations, and the restarted Lanczos method.
The main theoretical results on comparing the analytical rates of convergence between restarted Lanczos($m$) and momentum accelerated power iterations are presented in section \ref{sec:compare}. In section \ref{sec:precon} we discuss momentum accelerated power iterations as a preconditioner for restarted Lanczos($m$). In section \ref{sec:numerics} we present benchmark examples illustrating the theoretical results.

\subsection{Notation}\label{subsec:notation}
Suppose $A$ is an $n \times n$ real symmetric or complex Hermitian matrix with eigenvalues $\lambda_1 \ldots \lambda_n$, ordered by decreasing magnitude with $\lambda_1 > |\lambda_2| \ge \cdots \ge |\lambda_n|$, with corresponding eigenvectors $\phi_1, \ldots, \phi_n$. 
Note that we assume
there is a unique eigenvalue $\lambda_1$ of largest magnitude; 
however, the assumption that $\lambda_1 > 0$ is for notational convenience and clarity of presentation. If the largest magnitude eigenvalue of $A$ were negative, then the results as stated apply to $-A$.

Throughout the remainder of this paper, $n$ will refer to the dimension of matrix $A$, whereas $m$ will refer to the dimension of the Lanczos Krylov subspace $\mathcal K_m$.
We next state the base algorithms which we will refer to throughout the rest of the paper.
\subsection{Algorithms}\label{subsec:algo}
To fix notation, and later refer to its use as a preconditioner, we first state the power iteration.
\begin{algorithm}[Power]\label{alg:pow}
\begin{algorithmic}
\State{Choose $v_0$, set $h_0 = \nr{v_0}$ and $x_0 = h_0^{-1}v_0$}
\State{Set $v_1 = A x_0$}
\For{$k \ge 0$}
\State{Set $h_{k+1} = \nr{v_{k+1}}$ and  $x_{k+1} = h_{k+1}^{-1} v_{k+1}$}
\State{Set $v_{k+2} = A x_{k+1}$}
\State{Set $\nu_{k+1} = (v_{k+2}, x_{k+1})$ and 
       $\|d_{k+1}\| = \|v_{k+2} - \nu_{k+1} x_{k+1}\|$}
\State{{STOP} if $\nr{d_{k+1}} <$ \tt{tol}  }
\EndFor
\end{algorithmic}
\end{algorithm}

We next state the static momentum method with momentum parameter $\beta$, as introduced in \cite{XHDMR18}.
\begin{algorithm}[Static momentum]\label{alg:statmo}
\begin{algorithmic}
\State Set parameter $\beta >0$
\State{Do a single iteration of algorithm \ref{alg:pow}} \Comment{$k = 0$}
\For{$k \ge 1$} \Comment{$k \ge 1$}
\State{Set $u_{k+1} = v_{k+1} - (\beta/h_k) x_{k-1}$}
\State{Set $h_{k+1} = \nr{u_{k+1}}$ and $x_{k+1} = h_{k+1}^{-1} u_{k+1}$} 
\State{Set $v_{k+2} = A x_{k+1}$}
\State{Set $\nu_{k+1} = (v_{k+2}, x_{k+1})$ and 
           $\|d_{k+1}\| = \nr{v_{k+2} - \nu_{k+1} x_{k+1}}$}
\State{{STOP} if $\nr{d_{k+1}} <$ \tt{tol}  }
\EndFor
\end{algorithmic}
\end{algorithm}

As shown in \cite{APZ24,XHDMR18}, an efficient implementation of algorithm \ref{alg:statmo} requires a good approximation of $\lambda_2$ to approximate the optimal choice of $\beta$, which is $\lambda_2^2/4$.  When we use algorithm \ref{alg:statmo} as a preconditioner for restarted Lanczos, a good approximation to this parameter may be available.  Otherwise, the dynamic implementation introduced in \cite{APZ24} can be used to approximate $\lambda_2$, and hence the optimal momentum parameter $\beta$, by monitoring the residual convergence rate. The dynamic momentum algorithm stated next  was shown in \cite{APZ24} to achieve the efficiency of the optimal static method without explicit a priori knowledge of the spectrum.
\begin{algorithm}[Dynamic momentum]\label{alg:dymo}
\begin{algorithmic} 
\State {Do two iterations of algorithm \ref{alg:pow}} \Comment{$k = 0,1$}
\State{Set $r_{2} = \min\{d_2/d_{1},1\}$}
\For{$k \ge 2$} \Comment{$k \ge 2$}
\State{Set $\beta_k = \nu_{k}^2 r_{k}^2/4$}
\State{Set $u_{k+1} = v_{k+1} - (\beta_k/h_k) x_{k-1}$}
\State{Set $h_{k+1} = \nr{u_{k+1}}$ and $x_{k+1} = h_{k+1}^{-1} u_{k+1}$}
\State{Set $v_{k+2} = A x_{k+1}$}
\State{Set $\nu_{k+1} = (v_{k+2}, x_{k+1})$ and
           $\|d_{k+1}\| = \nr{v_{k+2} - \nu_{k+1} x_{k+1}}$}
\State{Set $\rho_k = \min\{\|d_{k+1}\|/\|d_k\|,1\}$  and
              $r_{k+1} = 2\rho_k/(1 + \rho_k^2)$}
\State{{STOP} if $\nr{d_{k+1}} <$ \tt{tol}  }
\EndFor
\end{algorithmic}
\end{algorithm}

Finally, we state the standard restarted Lanczos method, which we will refer to hereafter as restarted Lanczos($m$), to emphasize the dependence of its convergence on the dimension $m$ of the Krylov subspace used in its implementation.
\begin{algorithm}[Restarted Lanczos($m$)]\label{alg:lanczosm}
\begin{algorithmic}
\State{Choose $v_0$. Set $\beta_0 = \nr{v_0}$, $q_1 = v_0/\norm{v_0}$ and $q_{0} = 0$.}
\For {$1 \le k \le m-1$}
\State{Set $v = A q_k$} 
\State{Set $\alpha_k = (q_k,v)$}
\State{Set $u = v - \beta_{k} q_{k-1} - \alpha_k q_k$}
\State{Set $\beta_k = \norm{u}$}
\State{Set $q_{k+1} = u/\beta_k$}
\EndFor
\State{Set $v = A q_m$}
\State{Set $\alpha_m = (q_m,v)$}
\State{Set \[T = \begin{pmatrix} \ddots & \ddots & \ddots\\
                                 & \beta & \alpha & \beta\\
                                & & \ddots & \ddots & \ddots \end{pmatrix}\]}
\State{Compute the maximum magnitude eigenvalues $\nu_1, \nu_2$ of $T$, with corresponding eigenvectors $\tilde x_1, \tilde x_2$}
\State{Set $[x_1,x_2] = Q([\tilde x_1,\tilde x_2])$}
\State{Set $\|d_m\| = \|Ax_1 - \nu_1 x_1\|$}
\end{algorithmic}
\end{algorithm}

In the tests of sections \ref{subsec:ex1} and \ref{sec:numerics}, the 
condition to exit the loop is convergence of the residual, 
$\nr{d_{k+1}}<$ {\tt tol}, for a given tolerance {\tt tol}.

\section{Background}\label{sec:bg}
In this section we summarize the properties of the main algorithms to be considered: 
the static and dynamic versions of the momentum accelerated power iteration (algorithm \ref{alg:statmo} and \ref{alg:dymo}), and restarted Lanczos($m$) (algorithm \ref{alg:lanczosm}).  
First, we review some standard background on Chebyshev polynomials, which serve as the basis for the convergence analysis of the above methods.

\subsection{Chebyshev polynomials}\label{subsec:cheb}
The following results on the Chebyshev polynomials will be used throughout the remainder of the discussion. These results can be found, for instance, in \cite[Chapter 2]{Gautschi}, \cite[Chapter 2]{SaVi14}, and \cite[Chapter 4]{saad11-book}, which refers to \cite{cheney66-approx}.

The Chebyshev polynomials of the first kind satisfy the recurrence relation
\begin{align}\label{eqn:chebyrec}
T_{N+1}(t) = 2 t T_N(t) - T_{N-1}(t), \quad \text{for} \quad N \ge 1,
\end{align}
with $T_0(t) = 1$ and $T_1(t) = t$. In closed form, they are given by
\begin{equation}  \label{eqn:chebsqrtformula}
\begin{split}
T_N(t) &= \frac 1 2\left(  
   \left(t + \sqrt{t^2-1} \right)^N + \left( t - \sqrt{t^2-1}\right)^N \right)
   \\ &= \frac 1 2\left( 
\left(t + \sqrt{t^2-1} \right)^N + \left( t + \sqrt{t^2-1}\right)^{-N} \right),
\end{split}
\end{equation}
where the second formula follows from the fact that
$\left(t - \sqrt{t^2-1} \right) \left( t + \sqrt{t^2-1}\right) = 1$.
The Chebyshev polynomials also have the following explicit formula in terms of trigonometric and hyperbolic functions, namely
\begin{align}\label{eqn:chebycos}
T_N(x) = 
\begin{cases}
\cos(N \arccos(x)), & \text{for } x \in [-1,1] \\
\cosh(N \arccosh(x)), & \text{for } x \in (1,\infty) \\
(-1)^N \cosh(N \arccosh(-x)), & \text{for } x \in (-\infty,-1). \\
\end{cases}
\end{align}
From \eqref{eqn:chebycos}, each polynomial $T_N(t)$ satisfies $T_N(t) \le 1$ for $t \in[-1,1]$. 

Next we show that $T_N(t)$ grows exponentially 
as a function of $N$ outside of the interval $[-1,1]$, and state upper and lower bounds on the growth rate.
From \eqref{eqn:chebsqrtformula} we have
\begin{align*}
T_N(t) = \frac 1 2\left(  \left(t + \sqrt{t^2-1} \right)^N \!\!+ \left( t + \sqrt{t^2-1}\right)^{-N} \right) = \frac{1}{2} \left( t + \sqrt{t^2-1}\right)^N \left(1 + 
\left( t + \sqrt{t^2-1}\right)^{-2N} \right).
\end{align*}
Substituting $t = 1 + \varepsilon$ for some $\varepsilon > 0$ gives
\begin{align}\label{eqn:chub1}
T_N(1+\varepsilon) = \frac{1}{2} \left( 1 + \varepsilon + \sqrt{2\varepsilon + \varepsilon^2}\right)^N 
\!\!\left(1 + \left( 1 + \varepsilon + \sqrt{2 \varepsilon + \varepsilon^2}\right)^{-2N} \right).
\end{align}
Applying the inequalities
$1 + \sqrt{2 \varepsilon} \le  1 + \varepsilon + \sqrt{2\varepsilon + \varepsilon^2}$
and $\sqrt{2\varepsilon + \varepsilon^2} \le \sqrt{2 \varepsilon} + \varepsilon$, 
to \eqref{eqn:chub1} 
yields the lower and upper bounds
\begin{equation} \label{eqn:upperandlower}
\frac 1 2 \left( 1 + \sqrt{2 \eps}\right)^N \le
T_N(1 + \eps) \le 
\frac{1}{2} \left(1 + 2 \varepsilon + \sqrt{2 \varepsilon} \right)^N \left( 1 + (1 + \sqrt{2 \varepsilon})^{-2N} \right).
\end{equation}
Note that by symmetry, the same upper and lower bounds hold for $|T_N(-1-\varepsilon)|$.

Importantly for the discussion that follows, the Chebyshev polynomials satisfy an optimality property over all polynomials of fixed degree scaled to a value 1 at a point $\gamma$ that lies outside of an interval $[\alpha,\beta]$, given by
\begin{align}\label{eqn:chebyopt}
\min_{p \in \mathbb P_N, p(\gamma) = 1}~
\max_{t \in [\alpha,\beta]} |p(t)|
= \f{1}{\left| T_N\left(1 + 2\f{\gamma - \beta}{\beta - \alpha} \right) \right|},
\end{align}
where $\mathbb{P}_N$ denotes the space of polynomials of degree at most $N$.
\subsection{Momentum accelerated power iteration}\label{subsec:mop}
The standard power iteration $u_{k+1} = A x_k$, with normalization $x_k = u_k/\nr{u_k}$, for some norm $\nr{\cdot}$, is easily seen to satisfy $x_N = \widetilde C_N A^N x_0 = \widetilde C_N \widetilde p_N(A) x_0$, for normalization constant $\widetilde C_N$, and $\widetilde p_j(A) = A^j$. 
See, for instance, \cite[chapter 5]{QSS07}. A natural way to change the convergence properties of the iteration is to exchange the monomial $\widetilde p_j(\cdot)$, for some other polynomial $p_j(\cdot)$. 

If the polynomials $p_j(A)$ can be computed by a recurrence relation that involves one application of $A$ per iteration, then $p_N(A) x_0$ can be computed with essentially the same efficiency as  $\widetilde p_N(A) x_0$, so long as the coefficients of the recurrence relation for $p_N$ are available or can be computed easily. 
Next we consider how changing the iteration polynomial changes the convergence of the iteration.

As in the standard analysis of the power iteration, consider the expansion of the initial iterate $x_0$ as a linear combination of the eigenvectors of $A$, namely $x_0 = \sum_{i = 1}^n \alpha_i \phi_i$.
Then applying polynomial $p_j(A)$ we have
\begin{align}\label{eqn:polyexp}
p_j (A) x_0 =  \sum_{i = 1}^n \alpha_i p_j(A) \phi_i
=  \sum_{i = 1}^n \alpha_i p_j( \lambda_i) \phi_i.
\end{align}
For the iteration with $p_N(A)$ to converge faster than $\widetilde p_N(A)$, we require
\begin{align}\label{eqn:betterp}
\max_{i =2 \ldots n} \left| \f{p_N(\lambda_i)}{p_N(\lambda_1)} \right|
< \max_{i =2 \ldots n} \left| \f{\widetilde p_N(\lambda_i)}{\widetilde p_N(\lambda_1)} \right| = \left| \f{\lambda_2}{\lambda_1} \right|^N.
\end{align}
To connect \eqref{eqn:betterp} to
\eqref{eqn:chebyopt} 
in a way that requires only limited spectral knowledge
we can consider the problem of determining the polynomial $p \in \mathbb{P}_n$ that minimizes $\max_{t \in [-|\lambda_2|,|\lambda_2|]} |p(t)/p(\lambda_1)|$.  Taking
$\gamma = \lambda_1$, and $[\alpha,\beta] =
[-|\lambda_2|,|\lambda_2|]$, in \eqref{eqn:chebyopt} yields
\begin{align}\label{eqn:pNopt}
\min_{p \in \mathbb P_N, p(\gamma) = 1}~
\max_{t \in [-|\lambda_2|,|\lambda_2|]} |p(t)|
= \f{1}{\left| T_N\left(1 + \f{\lambda_1 - |\lambda_2|}{|\lambda_2|} \right) \right|}
= \f{1}{\left| T_N\left(\f{|\lambda_1|}{|\lambda_2|} \right) \right|}.
\end{align}
As the Chebyshev polynomials satisfy the optimality property \eqref{eqn:pNopt}, hence \eqref{eqn:betterp}, and are generated by a simple recurrence relation \eqref{eqn:chebyrec}, they are a natural choice to use in place of the monomials $\widetilde p_j(A)$ in a power-like iteration.  Next we look at how to build this idea into a method. 
In particular, we describe how to define a family of polynomials $p_N(x)$ that, after scaling, are optimal in the sense that they achieve the min-max of \eqref{eqn:pNopt}.

As shown in \cite{XHDMR18}, the polynomials $p_N(x)$ given by the recurrence relation
\begin{align}\label{eqn:hbrec}
p_{N+1}(x) = x p_N(x) - \beta p_{N-1}(x), 
\end{align}
with $p_0(x) = 1, p_1(x) = x/2$, satisfy the closed form
\begin{align}\label{eqn:hbclosed}
p_N(x) = \frac 1 2 \left( 
\left(\frac{x + \sqrt{x^2 - 4 \beta}}{2} \right)^N +  
\left(\frac{x - \sqrt{x^2 - 4 \beta}}{2} \right)^N
\right). 
\end{align}
The change of variables $x = 2 \sqrt\beta t$ then yields $p_N(x) = \beta^{N/2} T_N(t)$.
Applying these two relations to the Chebyshev recursion \eqref{eqn:chebyrec} yields
\begin{align*}
\beta^{-(N+1)/2} p_{N+1}(x) = x \beta ^{-(N+1)/2} p_N(x) - \beta^{-(N-1)/2}p_{N-1}(x),
\end{align*}
which, after multiplying through by $\beta^{(N+1)/2}$, is precisely \eqref{eqn:hbrec}.

This shows that the sequence of polynomials generated by \eqref{eqn:hbrec} is a scaled and stretched Chebyshev polynomial sequence, which remains bounded by one for 
$x \in [-2\sqrt \beta, 2 \sqrt \beta]$, and blows up outside that interval, satisfying the optimality property \eqref{eqn:chebyopt}.

Let $\beta  =  (\lambda_\ast/2)^2$, for $|\lambda_2| \le \lambda_\ast < \lambda_1$. Then from 
\eqref{eqn:chebycos}, we have 
\begin{align}\label{eqn:momrbeta}
\left| \frac {p_N(\lambda_j)}{p_N(\lambda_1)} \right|&= 
\left| \frac{T_N(\lambda_j/\lambda_\ast)}{T_N(\lambda_1/\lambda_\ast)} \right|
\le \left| \frac 1 {T_N(\lambda_1/\lambda_\ast)} \right|
= \left| \frac 1 {T_N\left(1 + \f{\lambda_1 - \lambda_\ast}{\lambda_\ast}\right)} \right|, ~\text{ for }~ j = 2 \ldots n,
\end{align}
which quantifies the convergence to the dominant eigenmode for iteration \eqref{eqn:hbrec}.
In particular, for $\lambda^\ast = |\lambda_2|$ we have
\begin{align}\label{eqn:momropt}
\frac {p_N(\lambda_2)}{p_N(\lambda_1)} 
= \frac{T_N(1)}{T_N\left(1 + \frac{\lambda_1 - |\lambda_2|}{|\lambda_2|} \right)}
= \frac{1}{T_N\left(1 + \frac{\lambda_1 - |\lambda_2|}{|\lambda_2|} \right)}
= \frac{1}{T_N\left(1 + \eps \right)},
~\text{ for }~ \eps \coloneqq \frac{\lambda_1 - |\lambda_2|}{|\lambda_2|},
\end{align}
which by \eqref{eqn:chebyopt} optimizes the convergence to the dominant eigenmode for \eqref{eqn:hbrec} with $\beta = \lambda_2^2/4$. Sharp results for the convergence of each subdominant eigenmode for $0 \le \lambda^\ast \le \lambda_1$ can be found in \cite{APZ24}, and are summarized below in subsection \ref{subsub:aug}. These results will be useful in subsection \ref{sec:precon}, where we consider the use of this iteration as a preconditioner.

\subsubsection{Normalized iteration}\label{subsub:hbnormalization} To develop a practical power-like iteration based on \eqref{eqn:hbrec}, we require its implementation based on one matrix-vector multiplication per iteration, and we require a normalization that prevents unbounded growth. To this end, let $C_N$ be a normalization factor at iteration $N$, and 
\begin{align}
x_N & = C_N p_N(A) x_0  \label{eqn:xn}\\
v_{N+1} & = Ax_N        \label{eqn:vn1}\\     
u_{N+1} & = v_{N+1} - \f{C_N}{C_{N-1}} \beta x_{N-1} \label{eqn:un1}.
\end{align}

Then, applying \eqref{eqn:xn} to \eqref{eqn:vn1}, we have
$
v_{N+1} = C_N A p_N(A) x_0,
$
and applying \eqref{eqn:xn}-\eqref{eqn:vn1} to \eqref{eqn:un1} allows
\begin{align}\label{eqn:un1p}
u_{N+1} 
= v_{N+1} - C_N \beta p_{N-1}(A) x_0 
= C_N (A p_N(A) - \beta p_{N-1}(A))x_0 
= C_N p_{N+1}(A) x_0,
\end{align}
where the last equality follows from \eqref{eqn:hbrec}. In light of the last equality in \eqref{eqn:un1p}, $x_N$ given in \eqref{eqn:xn} reduces to 
$(C_N/C_{N-1}) u_N$.
A natural choice of normalization is 
\begin{align}\label{eqn:normal}
C_N = \prod_{j = 1}^N h_j^{-1}, ~\text{ with }~ h_j = \norm{u_j},
\end{align}
by which \eqref{eqn:xn}-\eqref{eqn:un1} can be stated as
\begin{align}\label{eqn:mpscheme}
x_N & = h_N^{-1} u_N  \nonumber\\
v_{N+1} & = Ax_N        \nonumber\\     
u_{N+1} & = v_{N+1} - h_N^{-1} \beta x_{N-1},
\end{align}
which meets our requirement as a power-like iteration, and thanks to \eqref{eqn:momropt}, has an optimized rate of convergence, so long as $\beta$ can be defined appropriately. The static momentum algorithm \ref{alg:statmo} states \eqref{eqn:mpscheme} in algorithmic form. Next, we review results on the convergence of each eigenmode under \eqref{eqn:mpscheme}, then we review results on a dynamic implementation from \cite{APZ24} that allows the efficient approximation of the optimal constant $\beta = \lambda_2^2/4$.

\subsubsection{Convergence of each eigenmode}\label{subsub:aug}
While \eqref{eqn:momrbeta} and \eqref{eqn:momropt} quantify the convergence to the first eigenmode for iteration \eqref{eqn:mpscheme}, another approach to the analysis, in which we consider this iteration as a power iteration applied to an augmented matrix, will allow us to determine the convergence of each subdominant eigenmode to zero. These results will be useful in our discussion on preconditoning.  Moreover, this analysis shows how the method converges for any parameter $\beta \in (0,\lambda_1^2/4)$, and it shows that the method will not converge if $\beta \ge \lambda_1^2/4$. The following arguments were introduced in \cite{XHDMR18} and further detailed in \cite{APZ24}.

Given $n \times n$ matrix $A$, define the $2n \times 2n$ augmented matrix $A_\beta$ by 
\begin{align}\label{eqn:Abeta}
A_\beta = \begin{pmatrix} A & -\beta I \\ I & 0\end{pmatrix}.
\end{align}
Then the first component of the normalized iteration
\[
\begin{pmatrix} u_{k+1} \\ \widetilde u_{k+1}\end{pmatrix} = 
A_\beta \begin{pmatrix} x_{k} \\ \widetilde x_{k} \end{pmatrix},
\quad x_k = h_k^{-1} u_k, ~\widetilde x_k = h_k^{-1} \widetilde u_k,
\]
is seen to agree with the iteration given by \eqref{eqn:mpscheme}, for $h_k = \norm{u_k}$.  Thus the convergence of each mode follows from the standard analysis of the power iteration. It is sufficient to determine the eigenvalues $\mu_\lambda$ of $A_\beta$ corresponding to each eigenvalue of $\lambda$ of $A$, which as shown in \cite{APZ24,XHDMR18}, are given by
\begin{align}\label{eqn:mulambda}
\mu_{\lambda_{\pm}} = \frac 1 2 \left(\lambda \pm \sqrt{\lambda^2 - 4 \beta} \right).
\end{align}
In the case $\lambda^2 < 4 \beta$, \eqref{eqn:mulambda} reduces to 
$\mu_{\lambda_{\pm}} = (\lambda \pm i\sqrt{4 \beta - \lambda^2})/2$ which has the polar form representation 
\begin{align}\label{eqn:mupolar}
\mu_{\lambda_{\pm}}= r e^{i \theta}, ~\text{ with }~ r = \sqrt{\beta}
~\text{ and }~ \theta = \pm\arctan\left(\sqrt{4 \beta/\lambda^2-1} \right). 
\end{align}
Thus setting $\mu_{\lambda_j}$ as an eigenvalue of greatest magnitude corresponding to eigenvalue $\lambda_j$ of $A$, for each subdominant mode $j$, where $j = 2\ldots, n$, we have
\begin{align}\label{eqn:mpbymode}
\f{|\mu_{\lambda_j}|}{|\mu_{\lambda_1}|} = 
\left\{ \begin{array}{ll}
 \f{ |\lambda_j| +  \sqrt{\lambda_j^2 - 4\beta} }
    {|\lambda_1| +  \sqrt{\lambda_1^2 - 4\beta} }, 
& 0 \le \beta < \lambda_j^2/4 \\ \\
 \f{2\sqrt{\beta}}
   {|\lambda_1| +  \sqrt{\lambda_1^2 - 4\beta}}, 
& \lambda_j^2/4 < \beta < \lambda_1^2/4 
\end{array}\right. .
\end{align}
In the case that $\beta = \lambda_j^2/4$ for some eigenvalue $\lambda_j$ of $A$, the augmented matrix $A_\beta$ is actually defective. This is described in some detail in \cite{APZ24}. In particular, the optimal (asymptotic) convergence rate is recovered when $\beta = \lambda_2^2/4$ yielding
\begin{align}\label{eqn:amomrate}
\f{|\mu_{\lambda_2}|}{|\mu_{\lambda_1}|} \goto 
 \f{ |\lambda_2| }{|\lambda_1| + \sqrt{\lambda_1^2 - 4\beta} }
   = \f {r}{1 + \sqrt{1 - r^2}}, ~\text{ with }~ 
r = |\lambda_2 /\lambda_1| .
\end{align}
We emphasize that from \eqref{eqn:mupolar}, if $\lambda^2 < 4 \beta$, then the magnitude of each eigenvalue of \eqref{eqn:mulambda} is $\sqrt{\beta}$, 
and the eigenmodes with smaller eigenvalues feature increased phase angles, which will lead to increasingly oscillatory convergence. 
Moreover, if $\beta \ge \lambda_1^2/4$, then all eigenvalues of $A_\beta$ will have magnitude $\sqrt \beta$, and the method will not converge.

\subsubsection{Dynamic implementation}\label{subsub:hbdynamic}
The main issue preventing an efficient implementation of iteration \eqref{eqn:mpscheme}, i.e., algorithm \ref{alg:statmo}, with optimal parameter $\beta = \lambda_2^2/4$, is that $\lambda_2$ is generally {a priori} unknown. An approach for approximating $\lambda_2$ using a deflation technique is proposed in \cite{RJRH22}; however, this requires three matrix-vector multiplies per iteration in the deflation stage of the algorithm, and the guarantee of convergence following this algorithm depends on knowledge of $|\lambda_1-\lambda_2|$ and $|\lambda_2 - \lambda_3|$.
A more recent approach for approximating the optimal parameter $\beta$ is given in algorithm \ref{alg:dymo}. Here, the static $\beta$ is replaced by a dynamic approximation $\beta_k$ at each iteration.  The basis of the approximation is inverting the optimal accelerated rate $\rho(r) = r/(1 + \sqrt{1-r^2})$ as in \eqref{eqn:amomrate}, for $r(\rho)$, where $r = |\lambda_2/\lambda_1|$, yielding 
\begin{align}\label{eqn:rofrho}
r(\rho) = 2\rho/(1 + \rho^2).
\end{align}
Within the algorithm, $\rho_k$ is computed by the ratio of consecutive normed residuals, denoted $\|d_k\|$ in algorithm \ref{alg:dymo}, inverted via \eqref{eqn:rofrho} to determine $r_{k+1}$, and multiplied by the Rayleigh quotient to approximate $\lambda_2$, hence the optimal $\beta$. The essential mechanism by which $\beta_k$ starting from initial $\beta_0$ (computed from the ratio of residuals from two consecutive power iterations), tends towards the optimal $\beta = \lambda_2^2/4$, is that $r(\rho)$ is a contraction map.  In \cite{APZ24} we show this by a perturbation argument. 
This argument demonstrates that, when $r$ is close to $1$, the dynamic approximation becomes increasingly stable. Note that $r$ close to $1$ corresponds to the case where the spectral gap is small, which is the case of practical interest, where the power iteration is slow.

\begin{lemma}{\cite[Lemma 3]{APZ24}}\label{lem:rstab}
Let $\rho \in (0,1)$ and consider $\eps$ small enough so
that $(2 \rho \eps + \eps^2)/(1+\rho^2) < 1$.
Let $\rho_k = \rho + \eps$ and define
$r_{k+1} = 2 \rho_k/(1 + \rho_k^2)$, as in iteration \eqref{eqn:mpscheme} and algorithm \ref{alg:dymo}.
Then
\begin{align}\label{eqn:lemrstab}
r_{k+1} = r + \widehat \eps + \bigo(\eps^2) ~\text{ with }~
\widehat \eps = \eps \f{2 (1-\rho^2)}{(1+\rho^2)^2}.
\end{align}
\end{lemma}
The condition $(2 \rho \eps + \eps^2)/(1+\rho^2) < 1$ is satisfied for $\rho \in (0,1)$ by
$\eps < 0.71$. By \eqref{eqn:lemrstab}, $\widehat \eps/\eps$ tends to $0$ as $\rho$ tends to $1$. 

In the papers introducing dynamic iterations for the more general polynomial acceleration strategies of  \cite{CMP25-d,CMP25-r}, the map $r(\rho)$ given in this case
by \eqref{eqn:rofrho} is shown to be a contraction. Next we show the analogous result for iteration \ref{eqn:mpscheme}.

\begin{lemma}The map $r(\rho)$ given by \eqref{eqn:rofrho} is a contraction on
$(\sqrt{-2+\sqrt{5}},1) \approx (0.4859,1)$.
\end{lemma}
\begin{proof}
Taking a first derivative yields of $r(\rho)$ given by \eqref{eqn:rofrho} yields
\begin{align}\label{eqn:rprho}
r'(\rho) = 2(1-\rho^2)/(1+\rho^2)^2.
\end{align}
From \eqref{eqn:rprho}, we have $r'(\rho)$ is decreasing on $(0,1)$ with $r'(0) = 2$ and $r'(1) = 0$.  It suffices then to find $\rho$ in $(0,1)$ with $r'(\rho) =1$. This produces the quartic equation $\rho^4 + 4 \rho^2 -1 = 0$, which is easily solved using the quadratic formula with the substitution $y = \rho^2$, and produces a single real root in the interval $(0,1)$, namely $\rho = \sqrt{-2 + \sqrt{5}} \approx 0.4859$.
\end{proof}
From \eqref{eqn:rofrho}, the lower bound of the contraction region is
$\rho = \sqrt{-2 + \sqrt{5}}$, which yields $r(\rho) \approx 0.7862$. On one hand, the contraction result is exact, whereas the perturbation result neglects higher order terms which are negligible when $\eps$ is sufficiently small, but may not be negligible when $\eps$ is larger (even if $\eps < 0.71$ still holds).  However, the contraction result neglects the effect that $r(\rho)$ is increasing on $(0,1)$. In our computations, starting with $\rho_k$ outside of the contraction region (meaning $\rho_k$ too small) does not cause the dynamic approximation to fail; in practice, the sequence of approximations $r_k$ tend to increase into the contraction regime.

A full convergence result for the dynamic momentum algorithm \ref{alg:dymo} and comparison to the static momentum method algorithm \ref{alg:statmo} with optimal $\beta$ is given in \cite{APZ24}. Here, we will follow the conclusions of that paper, and run our numerical comparisons using the dynamic algorithm \ref{alg:dymo} when we are looking at the momentum accelerated power iteration as a standalone method.  When used in conjunction with restarted Lanzcos, we found it is more stable to use the output approximation to $\lambda_2$ from restarted Lanzcos in the static momentum algorithm \ref{alg:statmo}, instead.


\subsection{(Restarted) Lanczos}\label{subsec:lanczosm}
The Lanczos method generalizes the power iteration by building an orthonormal basis for the Krylov subspace for matrix $A$ with initial iterate $v_0$, given by $\mathcal K_m = \spa \{v_0, Av_0, \ldots, A^{m-1} v_0\}$. At each step $m$, the matrix of orthogonal basis vectors $Q_m$ satisfies $T_m = Q_m^T AQ_m$, where $T_m$ is tridiagonal. The eigenvalues of $T_m$ are the Rayleigh-Ritz approximations to the eigenvalues of $A$.  For symmetric matrices, the Lanczos method is mathematically equivalent to the Arnoldi method, which computes $Q_m^\ast A Q_m = H_m$, where $H_m$ is an upper Hessenberg matrix, and reduces to tridiagonal in the symmetric case. Computationally, the Lanczos method reduces the construction of the orthogonal basis held in $Q_m$ to a 3-term recurrence.

For analysis of Lanczos convergence, we rely on the bounds given by \cite{saad80,saad11-book}, which produce estimates of how the error decreases for a given spectrum as the Krylov subspace size $m$ increases.  Other important analyses of the Lanzcos process include \cite{KuWo89,urschel21}, 
which focus on spectrum-independent uniform convergence estimates based on matrix dimension, and may yield sharper results for general symmetric matrices, but not when the spectral gap is sufficiently small. 
The spectrum-dependent convergence theory for the Lanczos process is summarized in \cite[Theorem 6.3]{saad11-book}, see also \cite[Theorem 1]{saad80}.

\begin{remark}
In subsection \ref{subsec:notation} we ordered the eigenvalues with decreasing magnitude, with $\lambda_1 >|\lambda_2|$ and 
$|\lambda_j| \ge |\lambda_{j+1}|$, for $j \ge 2$. For the convergence of the Lanczos process, the analysis requires ordering the eigenvalues in a decreasing sequence. Under the assumption made in subsection \ref{subsec:notation} that $\lambda_1>0$ is the greatest magnitude eigenvalue, we introduce a relabeling of the eigenvalues, namely
\begin{align}\label{eqn:lanorder}
\lambda_{1'} > \lambda_{j'},  ~\text{ and }~ 
\lambda_{j'} \ge \lambda_{(j+1)'}, ~\text{ for }~ j' \ge 2.
\end{align}
Then we have $\lambda_{1'} = \lambda_1$. If $A$ is positive (semi-) definite then $\lambda_{j'} = \lambda_j$, for $j = 1, \ldots, n$.  For indefinite matrices  we may have $\lambda_{j'} \ne \lambda_j$, for $j \ge 2$. 
\end{remark}

\begin{theorem}{\cite[Theorem 6.3]{saad11-book}}.
\label{thm:lanczos-conv}
The angle between the eigenvector $\phi_{i'}$ associated with 
eigenvalue $\lambda_{i'}$ and the Krylov space 
$\mathcal K_m = \spa \{v_0, Av_0, \ldots, A^{m-1} v_0\}$,
satisfies the inequality
\begin{align}\label{eqn:lan-allvec}
\tan(\phi_{i'}, \mathcal K_m) \le \f{\kappa_i}{T_{m-i}(1 + 2 \gamma_{i'})}
\tan(x_0,\phi_{i'}),
\end{align}
where
\begin{align*}
\kappa_{i'} = \left\{ 
\begin{array}{lc}
1, & {i'} = 1 \\
\prod_{j' = 1}^{i'-1} \f{\lambda_{j'}-\lambda_{{n'}}}{\lambda_{j'}-\lambda_{i'}}, & i' > 1
\end{array}
\right., \quad ~\text{ and }~
\gamma_{i'} = \f{\lambda_{i'} - \lambda_{(i+1)'}}{\lambda_{(i+1)'}-\lambda_{n'}}.
\end{align*}
\end{theorem}

In particular, for $i' = 1$, Theorem \ref{thm:lanczos-conv} yields the estimate for convergence of the first (dominant) eigenmode
\begin{align}\label{eqn:lanerr}
\tan(\phi_1,\mathcal K_m) \le \f{1}{T_{m-1}\left(1 + 2 \eps_L \right) }\tan(x_0,\phi_1)
, ~\text{ with  }~ \eps_L \coloneqq \left(\f{\lambda_1 - \lambda_{2'}}{\lambda_{2'} - \lambda_{n'}}\right).
\end{align}
Understanding the convergence rate then reduces to quantifying how
the Chebyshev polynomial $T_{m-1}(1 + 2\eps_L)$ blows up for $\eps_L > 0$.
Applying \eqref{eqn:upperandlower} to \eqref{eqn:lanerr}
we have
\[
\left| T_{m-1}\left(1 + 2 
\left(\f{\lambda_1 - \lambda_{2'}}{\lambda_{2'} - \lambda_{n'}} \right)\right) \right|
=\left| T_{m-1}\left(1 + 2 \eps_L \right) \right|
\ge \frac 1 2 \left(1 + 2\sqrt{\eps_L} \right)^{(m-1)}.
\]
This yields an upper bound 
\begin{align}\label{eqn:lanupr}
\left| T_{m-1}\left(1 + 2 
\left(\f{\lambda_{1} - \lambda_{2'}}{\lambda_{2'} - \lambda_{n'}} \right)\right) \right|^{-1}
\le 2 \left( 1 + 2 \sqrt{\eps_L}\right)^{-(m-1)}.
\end{align}
For restarted Lanczos($m$), we can apply the bound \eqref{eqn:lanupr} repeatedly
to establish an upper bound on the convergence rate.
\section{Comparing rates}\label{sec:compare}
To determine regimes in which the momentum method 
with optimal parameter $\beta = \lambda_2^2/4$
converges faster than restarted 
Lanczos, we would like to determine the Krylov space dimension $m$ such that
the momentum accelerated convergence rate from \eqref{eqn:momropt} matches the Lanczos($m$) convergence rate determined by \eqref{eqn:lanupr}, that is
\begin{align}\label{eqn:rateq}
\log((T_{N+1}(1+\eps))^{-1}) - \log((T_N(1+\eps)))^{-1}) =
\frac { \log((T_{m-1}(1+2\eps_L))^{-2}) - \log((T_{m-1}(1+2\eps_L)^{-1} )}{m-1}.
\end{align}
Since the dynamic momentum algorithm \ref{alg:dymo}, which does not require knowledge of $\lambda_2$, has been shown to converge at the same rate, the results here apply to both methods.
Simplifying \eqref{eqn:rateq} yields
\begin{align}\label{eqn:rateqs}
\log(T_{N+1}(1+\eps)) - \log(T_N(1+\eps)) =
\frac {1} {m-1} \left( \log(T_{m-1}(1+2\eps_L)) \right),
\end{align}
where we expect the left hand side of \eqref{eqn:rateqs} to be independent of $N$ (for $N$ large enough), and the right hand side to depend $m$.  

For any fixed $\eps$, the ratio of upper and lower bounds for $T_N(1 + \eps)$ as provided by  \eqref{eqn:upperandlower} satisfy
\[ \f{ \frac{1}{2}\left( 1 + \varepsilon + \sqrt{2 \varepsilon + \varepsilon^2}\right)^N}
{\frac{1}{2} \left( 1 + \varepsilon + \sqrt{2\varepsilon + \varepsilon^2}\right)^N 
\left(1 + \left( 1 + \varepsilon + \sqrt{2 \varepsilon + \varepsilon^2}\right)^{-2N} \right)} \goto 1,
\]
as $N$ increases. Hence for sufficiently large $N$, 
namely, $N \gtrsim \eps^{-1/2}$, 
we can approximate each term in \eqref{eqn:rateqs} by the simpler estimate
$T_{N}(1+\eps) \approx \frac{1}{2}( 1 + \varepsilon + \sqrt{2 \varepsilon + \varepsilon^2})^N$. With this approximation, \eqref{eqn:rateqs} reduces to
\begin{align}\label{eqn:rateqa}
\log(1 + \eps + \sqrt{2\eps + \eps^2})
\approx
\f{-\log(2)}{m-1} + \log\left(1 + 2\eps_L + 2 \sqrt{\eps_L + \eps_L^2}\right).
\end{align}

Denote by $m_{cr}$ the $m$ at which the rates of Algorithms \ref{alg:statmo}/\ref{alg:dymo} and \ref{alg:lanczosm} cross. Assuming $\eps_L = \bigo(\eps)$, we can expand \eqref{eqn:rateqa} up to $\bigo\left( \sqrt{\eps} \right)$ and approximate $m_{cr}$ by
\begin{align}\label{eqn:mcr}
m_{cr} 
\approx 
\frac{\log(2)}{2 \sqrt{\eps_L} - \sqrt{2 \eps}  } + 1.
\end{align}

\begin{remark}\label{rem:aequiv}
In \cite{APZ24,XHDMR18}, and as summarized in subsection \ref{subsub:aug}, an explicit asymptotic convergence rate for the (static) momentum method with optimal parameter is found, namely
\[
\f{r}{1 + \sqrt{1-r^2}} 
~\text{ where }~ r = \f{|\lambda_2|}{\lambda_1} = \f{1 }{1+\eps},
\]
where $\eps$ is given by \eqref{eqn:momropt}. Using this rate in place of 
$\log((T_{N+1}(1+\eps))^{-1}) - \log((T_N(1+\eps)))^{-1})$ in \eqref{eqn:rateq}-\eqref{eqn:rateqs} yields
\begin{align}\label{eqn:conv01}
\log{\left( \left( \frac{r}{1 + \sqrt{1-r^2}}   \right)^{N+1} \right)}-
\log{\left( \left( \frac{r}{1 + \sqrt{1-r^2}}   \right)^{N} \right)}
 \log(r) - \log\left( \frac{r}{1 + \sqrt{1-r^2}}  \right).
\end{align}
Using $r = 1/(1 + \eps)$, by which
$\sqrt{1-r^2} = {1+\eps}^{-1} \sqrt{2\eps + \eps^2}$, \eqref{eqn:conv01}
reduces to $\sqrt{2 \eps} + \bigo(\eps)$, which agrees with the estimate for
$\log((T_{N+1}(1+\eps))^{-1}) - \log((T_N(1+\eps)))^{-1})$ used in \eqref{eqn:rateq}-\eqref{eqn:rateqs} to $\bigo\left(\sqrt{\eps}\right)$.
\end{remark}

\begin{remark}\label{rem:lanfast}
We observe in Example \ref{subsec:ex1}, the restarted Lanczos algorithm \ref{alg:lanczosm} converges at a rate close to the bound given by \eqref{eqn:lanerr} when $m <m_{cr}$ as given by \eqref{eqn:mcr}, but can converge faster than its predicted rate for $m$ large enough.  
\end{remark}
\subsection{Example 1}\label{subsec:ex1}
Let $A = \text{diag}(n:-1:1)$.  Then $\eps = 1/(n-1)$ and $\eps_L = 1/n$.
Table \ref{tab:computem} shows $m = m_{cr}$ such that restarted Lanczos($m$) attains the same approximate
convergence rate as the momentum-accelerated power method with the optimal parameter $\beta$. The table shows that the approximation \eqref{eqn:mcr} is sufficient to determine $m$ if $\eps$ and $\eps_L$ are known, and that $m$ increases linearly with $n$ for this example in a log-log scale. 

\begin{table}[]
    \centering
    \begin{tabular}{|l|c|c|c|}
        \hline
        $n$ & $\eps$ & $m_{cr}$ computed by \eqref{eqn:mcr} & $m_{cr}$ computed by \eqref{eqn:rateqs} and {\tt fzero}\\
        \hline
         128 & $7.87\times 10^{-3}$& 15 & 15 \\
         256 & $3.92\times 10^{-3}$& 20 & 20\\
         512 & $1.96\times 10^{-3}$& 28 & 28 \\
         1024 & $9.78\times 10^{-4}$& 39 & 39\\
         2048 & $4.89 \times 10^{-4}$& 55 & 54 \\
         4096 & $2.44 \times 10^{-4}$& 78 & 77 \\
         8192 & $1.22 \times 10^{-4}$& 109 & 107 \\
         16384& $6.10 \times 10^{-5}$& 153 & 151 \\
         \hline
    \end{tabular}
    \caption{The approximate value of $m$ at which \eqref{eqn:rateq} is satisfied, meaning restarted Lanczos($m$) and power-momentum converge at approximately the same rate for example 1 in subsection \ref{subsec:ex1}. In the right-hand column where \eqref{eqn:rateqs} is solved by the Matlab command {\tt fzero}, $N$ is taken to be 199 for $n \le 4096$, and 349 for $n \ge 8192$.}
    \label{tab:computem}
\end{table}

Figure \ref{fig:ratecompex1} shows reference lines for the predicted rates of convergence given in terms of the Chebyshev polynomials in \eqref{eqn:momropt} alongside the computed residual convergence of the dynamic momentum algorithm \ref{alg:dymo}, 
and the predicted rates of convergence \eqref{eqn:lanupr}, alongside the restarted Lanczos($m$) algorithm \ref{alg:lanczosm}, for $m = 8,16,32,64$.  
In each case $v_0$ was chosen as a vector of ones.
The plots show the dynamic momentum algorithm converging at its predicted rate, and restarted Lanczos($m$) converging at its predicted rate for $m = 8,16,32$, and faster than its rate given by the bound \eqref{eqn:lanupr}, for $m=64$. The plots also agree with table \ref{computem} which predicts the rates of convergence should agree at $m = m_{cr}=39$: in the bottom left plot we see the rates very close at $m=32$ with the dynamic momentum algorithm slightly faster, and in the bottom right with $m=64$, we see the restarted Lanczos($m$) algorithm faster than both the upper bound for its rate and the dynamic momentum method. 

\begin{figure}
\includegraphics[trim = 0pt 0pt 0pt 0pt,clip = true, width = 0.45\textwidth]
{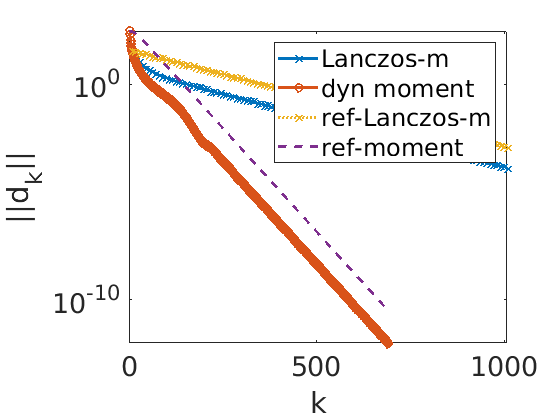}~\hfil~
\includegraphics[trim = 0pt 0pt 0pt 0pt,clip = true, width = 0.45\textwidth]
{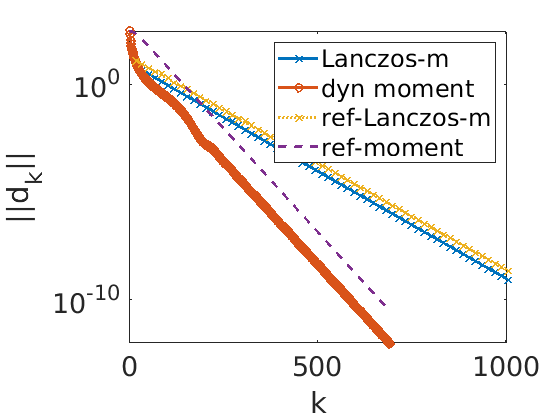} \\
\includegraphics[trim = 0pt 0pt 0pt 0pt,clip = true, width = 0.45\textwidth]
{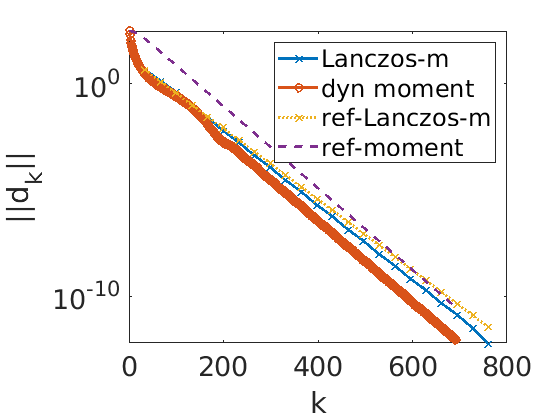}~\hfil~
\includegraphics[trim = 0pt 0pt 0pt 0pt,clip = true, width = 0.45\textwidth]
{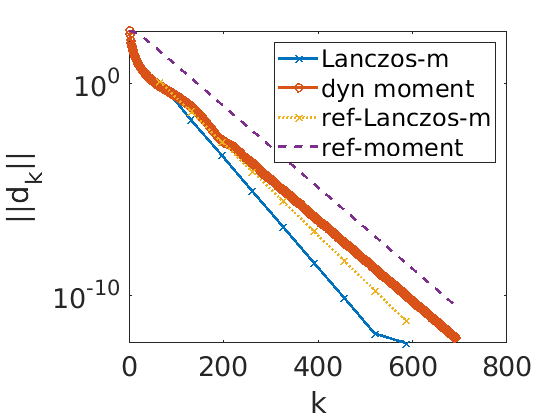} \\
\caption{Residual convergence of restarted Lanczos($m$) algorithm \ref{alg:lanczosm} and the dynamic momentum algorithm \ref{alg:dymo} for example 1, together with reference lines showing the predicted rates, with $n = 1024$ and
$m=8$ (top left), $m=16$ (top right), $m=32$ (bottom left), $m=64$ (bottom right).
\label{fig:ratecompex1}}
\end{figure}
\subsection{Example 2}\label{subsec:ex2}
Let $A = \text{diag}(n:-1:-n/2)$, for $n$ even.  Then $\eps = 1/(n-1)$ and $\eps_L = 1/(3n/2 -1)$.
Table \ref{tab:computem2} shows $m = m_{cr}$ such that restarted Lanczos($m$) attains the same approximate
convergence rate as the momentum-accelerated power method with the optimal parameter $\beta$. 
\begin{table}[]
    \centering
    \begin{tabular}{|l|c|c|c|}
        \hline
        $n$ & $\eps$ & $m_{cr}$ computed by \eqref{eqn:mcr} & $m_{cr}$ computed by \eqref{eqn:rateqs} and {\tt fzero}\\
        \hline
         128 & $7.87\times 10^{-3}$& 38 & 38 \\
         256 & $3.92\times 10^{-3}$& 52 & 52 \\
         512 & $1.96\times 10^{-3}$& 73 & 73 \\
         1024 & $9.78\times 10^{-4}$& 103 & 103\\
         2048 & $4.89 \times 10^{-4}$& 145 & 145\\
         \hline
    \end{tabular}
    \caption{The approximate value of $m$ at which \eqref{eqn:rateq} is satisfied, meaning restarted Lanczos($m$) and power-momentum converge at approximately the same rate for example 2 in subsection \ref{subsec:ex2}. In the right-hand column where \eqref{eqn:rateqs} is solved by the Matlab command {\tt fzero}, $N$ is taken to be 199.}
    \label{tab:computem2}
\end{table}

Comparing tables \ref{tab:computem} and \ref{tab:computem2} illustrates the dependence of $m_c$ on $\eps_L$ which may not agree closely with $\eps$ on indefinite problems. Convergence plots of this example are shown in section \ref{sec:numerics}, and include the use of momentum accelerated power iterations as an effective preconditioner or polynomial filter for restarted Lanczos($m$). In contrast to example 1 of subsection \ref{subsec:ex1}, where $m_c$ gives an accurate estimate of where the convergence rates cross, for the indefinite case, the predicted values of $m_c$ are higher than seen in our experiments.
\section{Preconditioning}\label{sec:precon}
To understand how we may use the momentum-based algorithm \ref{alg:statmo} as a 
preconditioner for restarted Lanczos, we consider how each of the subdominant eigenmodes in the first approximate eigenvector is reduced by the different algorithms.
As discussed in subsection \ref{subsub:aug},
the static momentum algorithm \ref{alg:statmo} with parameter $\beta = \lambda_\ast^2/4$ 
decays all subdominant modes, those with $|\lambda| < |\lambda_\ast|,$ at the same average rate, but with higher frequency oscillation for eigenvalues of decreasing magnitude. 
As shown in \cite{APZ24}, similar results hold for the dynamic momentum algorithm \ref{alg:dymo}. 

Figures \ref{fig:ratebymode1632} and \ref{fig:ratebymode64} show the average decay rate of each eigenmode for example 1 in subsection \ref{subsec:ex1} with $n = 1024$, for the power iteration, the dynamic momentum algorithm \ref{alg:dymo}, and restarted Lanczos($m$), with $m = 16, 32$ and $64$.
The average convergence rate for each mode was computed as the slope of the regression line found by Matlab's {\tt polyfit} function, fit to the residual convergence over an entire simulation. The smallest modes for the power iteration are not shown in the plots as they decay to zero very quickly after which their convergence rate is computed as {\tt NaN}.
As expected, we observe each mode in the power iteration decaying proportional to $\lambda/\lambda_1$, and each mode in the dynamic momentum algorithm \ref{alg:dymo} decaying at the same average rate. In contrast, the restarted Lanczos($m$) algorithm \ref{alg:lanczosm} displays oscillatory behavior with respect to $\lambda/\lambda_1$. We observe this behavior is not strictly periodic, but rather the frequency depends on 
the ratio $\lambda/\lambda_1$.  
We refer to this behavior as {\em quasi-periodic}. Figure \ref{fig:ratebymode1632} shows the restarted Lanczos($m$) algorithm with $m=16$ (left) and $m=32$ (right), and figure \ref{fig:ratebymode64} shows $m=64$, with a detail of the average decay rate of the larger subdominant modes.

\begin{figure}
\includegraphics[trim = 0pt 0pt 0pt 0pt,clip = true, width = 0.45\textwidth]
{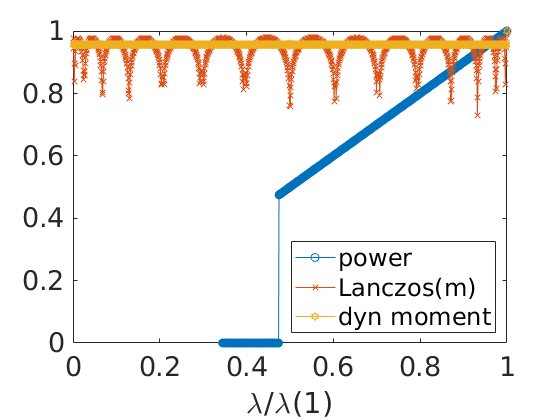}~\hfil~
\includegraphics[trim = 0pt 0pt 0pt 0pt,clip = true, width = 0.45\textwidth]
{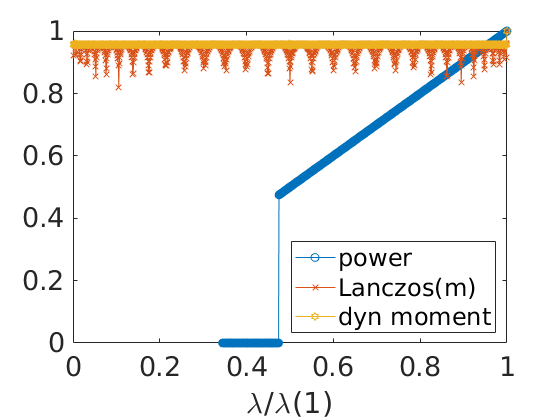} \\
\caption{Average reduction of subdominant eigenmodes by eigenvalue for the power, restarted Lanzcos($m$), and dynamic momentum algorithms applied to example 1 in subsection \ref{subsec:ex1} with $n = 1024$. Left: $m = 16$; right: $m=32$. The average convergence rate for each mode was computed as the slope of the regression line found by Matlab's {\tt polyfit} function.
\label{fig:ratebymode1632}}
\end{figure}

\begin{figure}
\includegraphics[trim = 0pt 0pt 0pt 0pt,clip = true, width = 0.45\textwidth]
{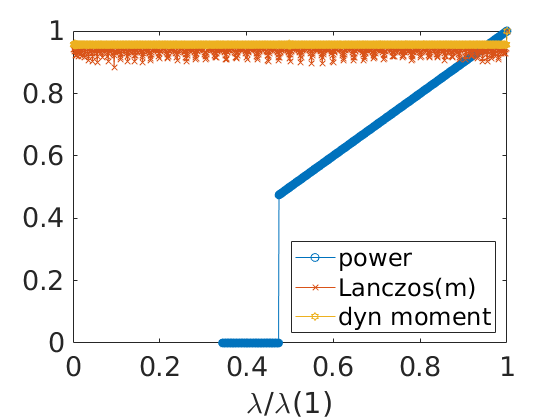}~\hfil~
\includegraphics[trim = 0pt 0pt 0pt 0pt,clip = true, width = 0.45\textwidth]
{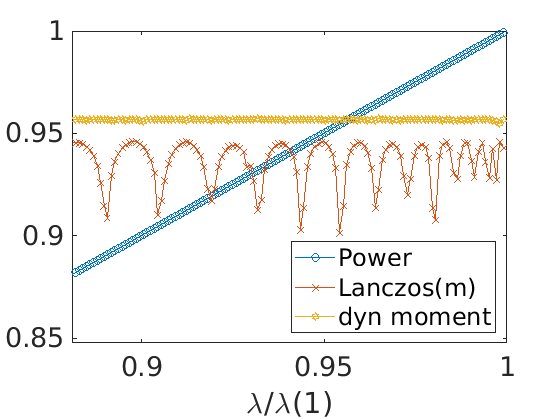} \\
\caption{Average reduction of subdominant eigenmodes by eigenvalue for the power, restarted Lanzcos($m$), and dynamic momentum algorithms applied to example 1 in subsection \ref{subsec:ex1} with $n = 1024$. Left: $m = 64$; right: $m=64$ (detail). The average convergence rate for each mode was computed as the slope of the regression line found by Matlab's {\tt polyfit} function.
\label{fig:ratebymode64}}
\end{figure}

In agreement with table \ref{tab:computem}, for $m = 16$ , the oscillations in the modal decay rate for algorithm \ref{alg:lanczosm} peak above the steady decay rate for algorithm \ref{alg:dymo}; for $m=32$, which is close to the predicted crossing point of $m=39$, the peaks for algorithm \ref{alg:lanczosm} are nearly aligned with the steady rate for algorithm \ref{alg:dymo}. In figure \ref{fig:ratebymode64} with $m=64$, we see the peaks for algorithm \ref{alg:lanczosm} are strictly beneath those for algorithm \ref{alg:dymo}, indicating a faster rate of convergence for all modes. This agrees with the faster residual convergence we see for this problem in figure \ref{fig:ratecompex1}. 

Next we consider preconditioning restarted Lanczos($m$) with momentum accelerated power iterations.  We additionally compare this approach with preconditioning with standard power iterations. In both cases, we are replacing the Krylov subspace $\mathcal K_m(x,A) = $span$\{x, Ax, \ldots A^{m-1}x\}$, with $\mathcal K_m(p_m(A)x,A))$, with $p_m(x)$, which is the shifted and rescaled Chebyshev polynomial satisfying \eqref{eqn:hbrec} and \eqref{eqn:hbclosed} in the momentum case; and $p_m(A)x = \widetilde C_m A^m x$ for preconditioning by the power iteration.  

We found experimentally the best performance using the balanced approach of $m$ momentum accelerated power iterations to precondition restarted Lanczos($m$).  Here we use the static momentum algorithm \ref{alg:statmo} rather than the dynamic algorithm \ref{alg:dymo}, as we can return an approximation to $\lambda_2$ as well as $\lambda_1$ from the restarted Lanczos algorithm \ref{alg:lanczosm}. The preconditioned algorithms are stated as follows. 

The restarted Lanczos($m$) algorithm preconditioned with momentum accelerated power iterations alternates one iteration of algorithm \ref{alg:lanczosm} with $m$ iterations of algorithm \ref{alg:statmo}.
\begin{algorithm}[Momentum preconditioned restarted Lanczos($m$)]\label{alg:mPL}
\begin{algorithmic}
\State{Choose $m$ and initial vector $x$}
\While{$\nr{d_{k+1}} \ge $ {\tt tol}}
\State{Run algorithm \ref{alg:lanczosm} with initial vector $v_0 = x$}
\State{-Output approximate eigenvalues and eigenvectors $\nu_1,\nu_2$ and $x_1,x_2$ }
\State{Run $m$ iterations of algorithm \ref{alg:statmo} with $\beta = \nu_2^2/4$ and $v_0 = x_1$} 
\State {-Output approximate eigenvector $x$}
\EndWhile
\end{algorithmic}
\end{algorithm}

We compare this method with restarted Lanczos($m$) preconditioned by standard power iterations, run by alternating one iteration of algorithm \ref{alg:lanczosm} with $m$ iterations of algorithm \ref{alg:pow}.
\begin{algorithm}[Power preconditioned restarted Lanczos($m$)]\label{alg:pPL}
\begin{algorithmic}
\State{Choose $m$ and initial vector $x$}
\While{$\nr{d_{k+1}} \ge $ {\tt tol}}
\State{Run algorithm \ref{alg:lanczosm} with initial vector $v_0 = x$}
\State{-Output approximate eigenvalue and eigenvector $\nu_1$ and $x_1$ }
\State{Run $m$ iterations of algorithm \ref{alg:pow} with $v_0 = x_1$} 
\State {-Output approximate eigenvector $x$}
\EndWhile
\end{algorithmic}
\end{algorithm}

Convergence to the dominant eigenvector $\phi_1$ is controlled by the slowest decaying subdominant eigenmode.  In the power iteration, this is naturally the second mode, as each mode decays like $|\lambda/\lambda_1|$. In the restarted Lanczos($m$) algorithm \ref{alg:lanczosm}, the slowest decay occurs quasi-periodically throughout the spectrum.
The role of preconditioning can be viewed as reducing the peak of the oscillations in the restarted Lanczos($m$) mode-wise decay rate. As in section \ref{sec:compare}, denote $m_{cr}$ as the critical value of $m$ where the convergence rate of restarted Lanczos($m$) algorithm \ref{alg:lanczosm} agrees with the optimal static or dynamic momentum decay rates of algorithms \ref{alg:statmo} and \ref{alg:dymo}. Then for $m < m_{cr}$, the momentum algorithms converge faster to the dominant eigenmode than restarted Lanczos($m$), and preconditioning the input yields an intermediate performance. 

Of particular interest, as seen in the examples in the next section, is that for $m > m_c$, the momentum preconditioned restarted Lanczos($m$) algorithm \ref{alg:mPL} converges in generally fewer iterations than  without the preconditioning. In contrast, preconditioning with standard power iterations as in algorithm \ref{alg:pPL} improves the iteration count less consistently. Figure \ref{fig:modes12-64} shows the detailed convergence for the 12th and 64 modes of example 1 of subsection \ref{subsec:ex1}, where restarted Lanczos($m$) is run with $m=64$ as in figure \ref{fig:ratebymode64}. The two modes we selected are at the peaks of the curve in figure \ref{fig:ratebymode64}, meaning the Lanczos convergence is the slowest. The plots illustrate how alternating with the momentum accelerated power iteration, which gives oscillatory convergence for each eigenmode, actually improves the average convergence rate for the slow-to-converge Lanczos modes.

In the next section we compare the restarted Lanczos and dynamic momentum algorithms \ref{alg:lanczosm} and \ref{alg:dymo} against the momentum preconditioned restarted Lanczos ($m$) algorithm \ref{alg:mPL} and the power preconditioned restarted Lanczos($m$) algorithm \ref{alg:pPL}, which only provides a convergence advantage to the eigenmodes with the smallest eigenvalues. We see the best overall convergence in most circumstances from the momentum preconditioned restarted Lanczos($m$) algorithm \ref{alg:mPL}, with $m$ chosen sufficiently large. 

\begin{figure}
\includegraphics[trim = 0pt 0pt 0pt 0pt,clip = true, width = 0.48\textwidth]
{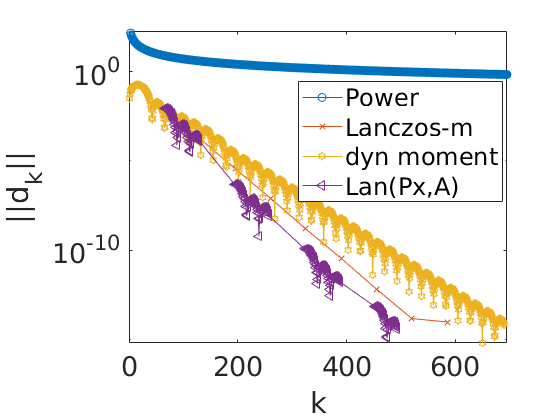}~\hfil~
\includegraphics[trim = 0pt 0pt 0pt 0pt,clip = true, width = 0.48\textwidth]
{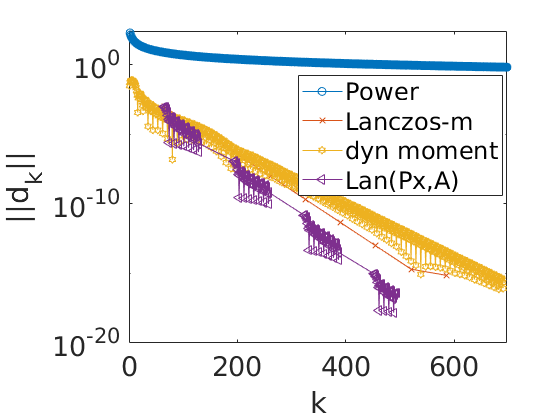} \\
\caption{Reduction of subdominant eigenmodes by matrix-vector multiplies $k$ for the power, restarted Lanzcos($m$), dynamic momentum, and preconditioned restarted-Lanczos($m$) algorithms for example 1 of subsection \ref{subsec:ex1} with $n = 1024$ and $m=64$. Left: mode $j=12$ with $\lambda_j/\lambda_1 = 1013/1024 \approx 0.9893$; right: mode $j=64$ with $\lambda_j/\lambda_1 = 961/1024 \approx 0.9385$
\label{fig:modes12-64}}
\end{figure}
\section{Numerical examples}\label{sec:numerics}
Here we compare convergence of the dynamic momentum algorithm \ref{alg:dymo}, the restarted Lanczos($m$) algorithm \ref{alg:lanczosm}, and restarted Lanczos($m$) using the static momentum algorithm \ref{alg:statmo} as in algorithm \ref{alg:mPL}, and the power iteration \ref{alg:pow} as a preconditioner as in algorithm \ref{alg:pPL}.

For the purpose of reproducibility of our results, we start each test with the vector of ones as the initial iterate.
We start with the indefinite matrix of example 2 from subsection \ref{subsec:ex2}, then consider sets of benchmark problems taken from the literature in subsections \ref{subsec:ts1} and \ref{subsec:ts2}.

\begin{figure}
\includegraphics[trim = 0pt 0pt 0pt 0pt,clip = true, width = 0.45\textwidth]
{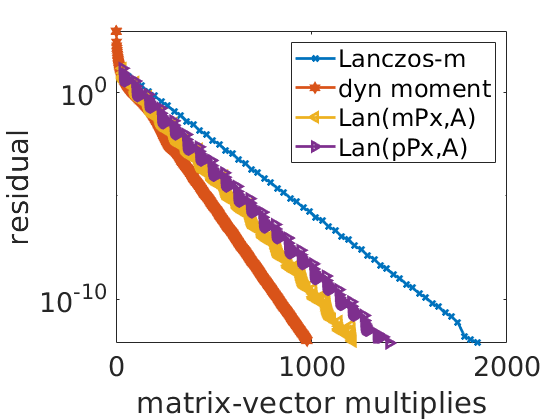}~\hfil~
\includegraphics[trim = 0pt 0pt 0pt 0pt,clip = true, width = 0.45\textwidth]
{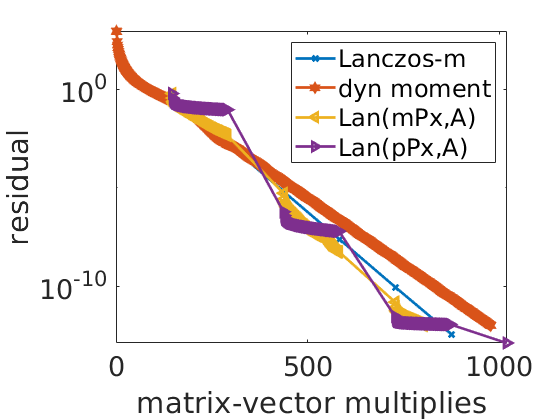} \\
\caption{Residual convergence of restarted Lanczos($m$) algorithm \ref{alg:lanczosm}, the dynamic momentum algorithm \ref{alg:dymo}, the momentum preconditioned Lanczos($m$) algorithm \ref{alg:mPL} denoted Lan(mPx,A), and the power preconditioned Lanczos($m$) algorithm denoted Lan(pPx,A), for example 2 of subsection \ref{subsec:ex2}, with $n = 2048$ and
$m=32$ (left), and $m=145$ (right).
\label{fig:ex2indef}}
\end{figure}

In figure \ref{fig:ex2indef}, we show residual convergence for restarted Lanczos($m$) algorithm \ref{alg:lanczosm}, the dynamic momentum algorithm \ref{alg:dymo}, the momentum preconditioned Lanczos($m$) algorithm \ref{alg:mPL} denoted Lan(mPx,A), and the power preconditioned Lanczos($m$) algorithm denoted Lan(pPx,A), for example 2 of subsection \ref{subsec:ex2}. Here, $A$ is the $3073 \times 3073$ matrix given by $A = \diag(n:-1:-n/2)$, with $n = 2048$.
The dimension of the Krylov subspace used for restarted Lanczos in the 
left plot is $m=32$, and the right plot shows $m=145$. 
In subsection \ref{subsec:ex2} the predicted Krylov subspace dimension $m_c$ where algorithms \ref{alg:dymo} and \ref{alg:lanczosm} would converge at the same rate was $m_c = 145$.  As seen in figure \ref{fig:ex2indef}, this appears to be an overestimate as the 
convergence plot shows an observably steeper slope for restarted Lanczos($m$) than for dynamic momentum with $m=145$. It is further interesting to notice in this plot that the preconditioning stage shows a steeper slope using the momentum method than the power method, however the resulting decrease in preconditioned Lanzcos is more pronounced with the power method preconditioning with $m=145$.  However, in the left plot of figure \ref{fig:ex2indef} where $m=32$ which is well within the regime where the dynamic method converges faster, we see the momentum preconditioning gives an overall faster convergence then the power preconditioning. 

\subsection{Test set 1}\label{subsec:ts1}
Our first set of test matrices given in table \ref{tab:testset1} is from \cite{WuSi00}. 
\begin{table}[]
    \centering
    \begin{tabular}{|l|c|c|c|c|}
        \hline
        matrix & $n$ & $\lambda_1$  &$\eps$  \\
        \hline
         {\tt nasasrb} &54,870 & $2.64806 \times 10^9$ & $3.68692\times10^{-5}$\\
         {\tt s3dkq4m2} &90,449 &$4601.65$ &$2.24013\times10^{-4}$ \\
         {\tt s3dkt3m2} &90,449 &$8798.44$ &$4.33693\times10^{-4}$ \\
         \hline
    \end{tabular}
    \caption{Test set 1, the matrices from \cite{DH11} used in \cite{WuSi00}. All three matrices in this set are symmetric positive definite (SPD). The values
    of $\lambda_1$ and $\eps$ listed for each matrix are computed from algorithm \ref{alg:lanczosm} with $m=64$, and agreed across all methods that sufficiently converged.}
   \label{tab:testset1}
\end{table}
\begin{table}[]
    \centering
    \begin{tabular}{|l|c|c|c|}
        \hline
           & {\tt nasasrb} & {\tt s3dkq4m2} & {\tt s3dkt3m2} \\
        \hline
    dyn moment & 2431 & 942 & 1205 \\
    \hline
    Lanczos(16)&$>$5000 & 2635& 4267\\
    Lanczos(32)&462 & 1452& 2211\\
    Lanczos(64)&{\bf 195} & 910 & 1365\\
    \hline
    mp-Lanczos(16)& 3630& 1485 & 2508\\
    mp-Lanczos(32)& 455& 1105 & 1625\\
    mp-Lanczos(64)&258 & {\bf 774} & {\bf 1032}\\
    \hline
    pp-Lanczos(16)&4884 & 1881 & 2970\\
    pp-Lanczos(32)& 455& 1430& 2015 \\
    pp-Lanczos(64)&258 & 1161& 1419 \\
         \hline
    \end{tabular}
    \caption{Results for the test set 1 matrices given in \ref{tab:testset1}. The number of matrix-vector multiplies to (relative) residual convergence $\|Ax - \lambda x\|/|\lambda| < 10^{-12}$. Algorithm \ref{alg:dymo} is denoted dyn moment, algorithm \ref{alg:lanczosm} is denoted Lanczos($m$), algorithm \ref{alg:mPL} is denoted mp-Lanczos($m$) and algorithm \ref{alg:pPL} is denoted pp-Lanczos($m$.)}
   \label{tab:set1results}
\end{table}
Due to the scaling of the eigenvalues of the first matrix {\tt nasarb}, which features $\lambda_1$ on the order of $10^{9}$, we terminated iterations once the relative residual $\|Ax - \nu x\|/|\nu| <${\tt tol}. 
On this test set, we generally saw the best convergence by iteration count with momentum preconditioned Lanzcos($64$), where we tested Krylov subspace sizes $m = 16,32,64$. The exception to this is {\tt nasarb}, which converged in 195 iterations without preconditioning and 258 with.  In agreement with our results in section \ref{sec:compare}, we also see the dynamic momentum algorithm \ref{alg:dymo} outperforms restarted Lanczos($m$) algorithm \ref{alg:lanczosm} for smaller values of $m$, but restarted Lanczos($m$) converges faster after some critical value $m > m_c$.  Here we did not attempt to compute $m_c$ as $\eps_L$ given by \eqref{eqn:lanerr} is unknown without further spectral knowledge.

\subsection{Test set 2}\label{subsec:ts2}
Our second set of test matrices from \cite{DH11} is used to benchmark performance in \cite{DSYG18}. Each of these matrices is symmetric. Table \ref{tab:testset2} specifies whether each matrix is additionally positive definite (SPD), and whether it is labeled as a Cholesky candidate in \cite{DH11}.  The values of $\lambda_1$ listed agreed across all methods.

\begin{table}[]\label{computem}
    \centering
    \begin{tabular}{|l|c|c|c|c|}
        \hline
        matrix &$n$ & $\lambda_1$ & SPD & Cholesky candidate  \\
        \hline
         {\tt Si5H12}      & 19896& 58.5609 & N & Y\\
         {\tt c-65}        & 48066& 131413 & N & N\\
         {\tt Andrews}     & 60000&36.4853 & Y & Y \\
         {\tt Ga3As3H12}   & 61349& 1299.88 & N & Y\\
         {\tt Ga10As10H30} & 113081& 1300.8 & N & Y \\
         \hline
    \end{tabular}
    \caption{Test set 2, the matrices from \cite{DH11} used in \cite{DSYG18}. All matrices in this set are symmetric. The values
    of $\lambda_1$ listed for each matrix agreed across all methods.}
   \label{tab:testset2}
\end{table}

\begin{table}[]
    \centering
    \begin{tabular}{|l|c|c|c|c|c|}
        \hline
           & {\tt Si5H12} & {\tt c-65} & {\tt Andrews} & {\tt Ga3As3H12} & {\tt Ga10As10H30} \\
        \hline
    dyn moment &456 &40 & 592& $F(9\!\cdot\!10^{-9})$ & F($3\!\cdot\!10^{-5}$)\\
    \hline
    Lanczos(16)&527 & F($5\!\cdot\!10^{-12}$)& 1071& $F(2\!\cdot\!10^{-12})$ & F($3\!\cdot\!10^{-5}$)\\
    Lanczos(32)&330 & F($2\!\cdot\!10^{-12}$)& 528& 231 & 429\\
    Lanczos(64)&325 & F($5\!\cdot\!10^{-12}$)& 260& F($2\!\cdot\!10^{-12}$)& 585\\
    \hline
    mp-Lanczos(16)&460 &{\bf24} &650 & 184& F($2\!\cdot\!10^{-4}$) \\
    mp-Lanczos(32)&{\bf322} &101 &511 &{\bf165} & 425 \\
    mp-Lanczos(64)&325 &83 & 325& 327& {\bf196}\\
    \hline
    pp-Lanczos(16)&481 &26 & 646& 184 & F($2\!\cdot\!10^{-4}$)\\
    pp-Lanczos(32)&361 &40 & 490 & 167& 425 \\
    pp-Lanczos(64)&324 &73 & {\bf244}& 196&{\bf196} \\
         \hline
    \end{tabular}
    \caption{Results for the test set 2 matrices given in \ref{tab:testset1}. The number of matrix-vector multiplies to (absolute) residual convergence $\|Ax - \lambda x\|< 10^{-12}$. Runs that converged to a residual tolerance greater than $10^{-12}$ after 5000 iterations are denoted F, and the approximate tolerance they stabilized to is listed. Algorithm \ref{alg:dymo} is denoted dyn moment, algorithm \ref{alg:lanczosm} is denoted Lanczos($m$), algorithm \ref{alg:mPL} is denoted mp-Lanczos($m$) and algorithm \ref{alg:pPL} is denoted pp-Lanczos($m$.)}
   \label{tab:set2results}
\end{table}

The results shown in table \ref{tab:set2results} show the momentum preconditioned Lanczos($m$) algorithm \ref{alg:mPL} converging with the fewest matrix-vector multiplies for 
three of the matrices in the set, with the power preconditioned version \ref{alg:pPL} matching that efficiency in one case and exceeding it in another.  
Of interest, the larger Krylov subspace sizes $m$ did not provide an advantage in all of these test cases.  In particular, for {\tt c-65}, the fastest convergence was found with $m=16$.
For this matrix, each run of restarted Lanczos without preconditioning terminated after a maximum number of iterations having nearly but not quite achieved the desired tolerance of $10^{-12}$. This demonstrates that the preconditioning can allow increased accuracy.
Similar results were found for {\tt Ga3As3H12}, which achieved the best convergence with momentum preconditioned algorithm \ref{alg:mPL} using $m=32$.
For {\tt Si5H12}, the fastest convergence was found with $m=32$, using momentum preconditioning although the iteration count differs by only three between $m=32$ and $m=64$, and similar to power preconditioning with $m=64$.  One of the interesting observations from this data set is that the momentum preconditioning can be advantageous for restarted Lanczos for smaller Krylov subspace sizes, achieving the same or better convergence than either restarted Lanczos($m$) alone or with power preconditioning with larger Krylov subspace sizes.  Overall this points to efficiency advantages in both iteration count and memory.

\section{Conclusion}
In this paper we quantified the comparative efficiency between a momentum accelerated power iteration and the restarted Lanczos algorithm for symmetric (or complex Hermitian) eigenvalue problems. We then investigated the momentum accelerated power iteration as a polynomial filter or preconditioning technique to accelerate restarted Lanczos. As shown in \cite{APZ24}, while each subdominant mode in the momentum accelerated power iteration converges at the same average rate, the convergence of modes with smaller eigenvalues is more oscillatory. Here we also illustrate that the average convergence of the restarted Lanczos method is quasi-periodic with respect to eigenvalue magnitude.  
We found the momentum accelerated power iteration, which can be efficiently implemented using the dynamic strategy of \cite{APZ24}, has a faster convergence rate that restarted Lanczos($m$) for $m < m_c$, where $m_c$ depends on the relative spectral gap $\eps = (\lambda_1 - \lambda_2)/\lambda_2$, and also $\eps_L = (\lambda_1 - \lambda_{2'})/(\lambda_{2'} - \lambda_{n'})$, where $\lambda_2$ is the eigenvalue of secondary magnitude, $\lambda_{2'}$ is the second largest signed eigenvalue, and $\lambda_{n'}$ is the left-most eigenvalue in the spectrum of $A$, which agrees with $\lambda_n$ when $A$ is positive definite, but not when $A$ has negative eigenvalues.

Our numerical results confirmed our theoretical results on estimating which method would be more efficient based on a priori knowledge of $\eps$ and $\eps_L$.  Our numerical results further confirmed that momentum accelerated power iterations are an effective preconditioner for the restarted Lanczos algorithm, improving both the number of matrix-vector multiplies to convergence, and residual accuracy. This suggests it may be promising to study the analogous technique for classes of nonsymmetric matrices using the Arnoldi method in place of Lanczos, and the generalized momentum methods of \cite{CMP25-d,CMP25-r} in place of the Chebyshev-based momentum methods.

\section{Acknowledgements}
AB and SP were supported in part by NSF grant DMS 2045059 (CAREER).


\bibliographystyle{abbrv}
\bibliography{accrefs}

@article{sorensen1992implicit,
  title={Implicit application of polynomial filters in a k-step {A}rnoldi method},
  author={Sorensen, Danny C},
  journal={Siam journal on matrix analysis and applications},
  volume={13},
  number={1},
  pages={357--385},
  year={1992},
  publisher={SIAM}
}

@article{KuWo89,
  author = {Kuczy\'{n}ski, J. and Wo\'{z}niakowski, H.},
  title  = {Estimating the Largest Eigenvalue by the Power and {L}anczos Algorithms with a Random Start},
  journal = {SIAM Journal on Matrix Analysis and Applications},
  volume  = {13},
  number  = {4},
  pages   = {1094-1122},
  year    = {1992},
  doi     = {10.1137/0613066},
}

@article{knyazev87,
  author = {A. V. Knyazev},
  title  = {Convergence rate estimates for iterative methods for a mesh symmetric eigenvalue problem}, 
  journal = {Russian J. Numer. Anal. Math. Modelling}, 
  year    = {1987},
  volume  = {2},
  number  = {5},
  pages   = {371–396},
}

@misc{CMP25-d,
  author = {P. Cowal and N. F. Marshall and S. Pollock}, 
  title  = {Faber polynomials in a deltoid region and power iteration momentum methods}, year   = {2025},
  note   = {https://arxiv.org/abs/2507.01885}
}

@misc{CMP25-r,
  author = {P. Cowal and  N. F. Marshall and S. Pollock}, 
  title  = {Random walks, {F}aber polynomials, and accelerated power methods}, 
  year   = {2025}, 
  note   = {https://arxiv.org/abs/2510.24608} 
}

@book{QSS07,
  title    = {Numerical Mathematics},
  author   = {A. Quarteroni and R. Sacco and F. Saleri},
  publisher= {Springer},
  series   = {Texts in Applied Mathematics},
  year     = {2007},
  location = {Berlin Heidelberg},
  edition  = {2},
  isbn     = {978-3-540-34658-6}
}

@article{GuRo02,
title   = {The {C}hebyshev iteration revisited},
journal = {Parallel Computing},
volume  = {28},
number  = {2},
pages   = {263-283},
year    = {2002},
doi     = {10.1016/S0167-8191(01)00139-9},
author  = {Martin H. Gutknecht and Stefan Röllin},
}

@InProceedings{RJRH22,
  title =        {Practical and Fast Momentum-Based Power Methods},
  author =       {Rabbani, Tahseen and Jain, Apollo and Rajkumar, Arjun and Huang, Furong},
  booktitle =    {Proceedings of the 2nd Mathematical and Scientific Machine Learning Conference},
  pages =        {721-756},
  year =         {2022},
  editor =       {Bruna, Joan and Hesthaven, Jan and Zdeborova, Lenka},
  volume =       {145},
  series =       {Proceedings of Machine Learning Research},
  publisher =    {PMLR},
  url =          {https://proceedings.mlr.press/v145/rabbani22a.html},
}

@article{DSYG18,
  author  = {Duersch, Jed A. and Shao, Meiyue and Yang, Chao and Gu, Ming},
  title   = {A Robust and Efficient Implementation of {LOBPCG}},
  journal = {SIAM Journal on Scientific Computing},
  volume  = {40},
  number  = {5},
  pages   = {C655-C676},
  year    = {2018},
  doi     = {10.1137/17M1129830},
}

@article{DH11,
  author  = {Timothy A. Davis and Yifan Hu},
  title   = {The {U}niversity of {F}lorida Sparse Matrix Collection},
  journal = {ACM Transactions on Mathematical Software},
  volume  = {38},
  number  = {1},
  pages   = {1-25},
  year    = {2011},
  doi     = {10.1145/2049662.2049663},
  url     = {https://sparse.tamu.edu}
}

@article{aishima15,
  author = {Aishima, K.}, 
  title = {Global convergence of the restarted {L}anczos and {J}acobi–{D}avidson methods for symmetric eigenvalue problems}, 
  journal = {Numer. Math.}, 
  volume = {131}, 
  pages  = {405–423},
  year   = {2015},
  doi    = {10.1007/s00211-015-0699-4}
}

@article{urschel21,
author = {Urschel, John C.},
title = {Uniform Error Estimates for the {L}anczos Method},
journal = {SIAM Journal on Matrix Analysis and Applications},
volume = {42},
number = {3},
pages = {1423-1450},
year = {2021},
doi = {10.1137/20M1331470},
}

@article{WuSi00,
  author = {Wu, Kesheng and Simon, Horst},
  title = {Thick-Restart {L}anczos Method for Large Symmetric Eigenvalue Problems},
  journal = {SIAM Journal on Matrix Analysis and Applications},
  volume = {22},
  number = {2},
  pages = {602-616},
  year = {2000},
  doi = {10.1137/S0895479898334605},
}

@book{Gautschi,
  author = {Walter Gautschi},
  title  = {Numerical Analysis},
  publisher = {Birkhäuser},
  isbn      = {978-0-8176-8259-0},
  address   = {Boston, MA},
  edition   = {2},
  year      = {2011}
}

@book{cheney66-approx,
  author = {E. W. Cheney},
  title  = {Introduction to Approximation Theory},
  edition= {1},
  publisher= {McGraw-Hill},
  year   = {1966},
  ISBN   = {9780070107571},
  address= {}
}

@article{saad80,
  author = {Y. Saad},
  title  = {On the rates of convergence of the {L}anczos and the block-{L}anczos methods},
    journal={SIAM Journal on Numerical Analysis},
  volume={17},
  number={5},
  pages={687--706},
  year={1980},
  publisher={SIAM} 
}

@book{saad11-book,
author = {Y. Saad},
title = {Numerical Methods for Large Eigenvalue Problems},
publisher = {Society for Industrial and Applied Mathematics},
year = {2011},
doi = {10.1137/1.9781611970739},
address = {},
edition   = {},
}

@article{GoVa61,
  title={Chebyshev semi-iterative methods, successive overrelaxation iterative methods, and second order {R}ichardson iterative methods},
  author={Golub, Gene H and Varga, Richard S},
  journal={Numerische Mathematik},
  volume={3},
  number={1},
  pages={157--168},
  year={1961}
}

@article{polyak1964,
  title={Some methods of speeding up the convergence of iteration methods},
  author={Polyak, Boris T},
  journal={USSR computational mathematics and mathematical physics},
  volume={4},
  number={5},
  pages={1--17},
  year={1964},
  publisher={Elsevier}
}

@article{PoSc21-arnoldi,
  title={Extrapolating the {A}rnoldi algorithm to improve eigenvector convergence},
  author={Pollock, Sara and Scott, L Ridgway},
  journal={International Journal of Numerical Analysis and Modeling},
  volume = {18},
  number = {5},
  pages  = {712--721},
  year={2021}
}

@article{APZ24,
  title={Dynamically accelerating the power iteration with momentum},
  author={Austin, Christian and Pollock, Sara and Zhu, Yunrong},
  journal={Numerical Linear Algebra with Applications},
  volume={31},
  number={6},
  pages={e2584},
  year={2024},
  publisher={Wiley Online Library}
}

@inproceedings{XHDMR18,
  title={Accelerated stochastic power iteration},
  author={Xu, Peng and He, Bryan and De Sa, Christopher and Mitliagkas, Ioannis and Re, Chris},
  booktitle={International Conference on Artificial Intelligence and Statistics},
  pages={58--67},
  year={2018},
  organization={PMLR}
}

@article{SaVi14,
  title={Faster algorithms via approximation theory},
  author={Sachdeva, Sushant and Vishnoi, Nisheeth K and others},
  journal={Foundations and Trends{\textregistered} in Theoretical Computer Science},
  volume={9},
  number={2},
  pages={125--210},
  year={2014},
  publisher={Now Publishers, Inc.}
}

\end{document}